\documentclass[a4paper, ariel, 11pt]{amsart}
\usepackage{amscd,amsthm,amsfonts,latexsym,amssymb}
\usepackage[dvips]{graphicx}

\usepackage{bez123,calc,curves,ebezier,epic,eepic,graphicx,multiply,rotating}

\newtheorem{theorem}{Theorem} [section]
\newtheorem{lemma}[theorem]{Lemma}

\newtheorem{proposition}[theorem]{Proposition}

\newtheorem{definition}[theorem]{Definition}

\newcommand{\nks}{\ensuremath{\mathrm{SL(2,\mathbb{R})}\times \mathrm{SL(2,\mathbb{R})}}} 

\renewcommand{\phi}{\varphi}

\begin{document}

\title{Almost complex surfaces in the nearly K\"{a}hler  $\mathbf{SL(2,\mathbb{R})\times SL(2,\mathbb{R})}$}
\author{Elsa Ghandour and Luc Vrancken}

\address{Elsa Ghandour, Faculty of Science\\
University of Lund\\
Box 118, Lund 221\\
Sweden}
\email{Elsa.Ghandour@math.lu.se}

\address{Luc Vrancken, LMI-Laboratoire de Math\'ematiques pour l'Ingénieur \\
Universit\'e Polytechnique Hauts-de-France \\
Campus du Mont Houy  \\
59313 Valenciennes Cedex 9, France}

\email{luc.vrancken@uphf.fr}

\begin{abstract} 
The space $SL(2,\mathbb{R})\times SL(2,\mathbb{R})$ admits a natural homogeneous pseudo-Riemannian nearly Kaehler structure. We investigate almost complex surfaces in this space. In particular we obtain a complete classification of the totally geodesic almost complex surfaces and of the almost complex surfaces with parallel second fundamental form. 
\end{abstract}



\maketitle

\thispagestyle{empty}


\section{Introduction}   

Almost complex structures were introduced by C. Ehresmann in 1947 \cite{Eh1} in order to study the problem of finding differentiable manifolds which admit a complex analytic structure related to the differentiable structure on the manifold. Approaching the same problem but by using different methods, H. Hopf has shown that the spheres $S^4$ and $S^8$ do not have any complex structure \cite{Ho}. Moreover, A. Borel and J.P. Serre \cite{Bo-Se} have shown that $S^2$ and $S^6$ are the only spheres admitting almost complex structures. Up to today it is still an open problem to determine if on $S^6$ there exists an integrable almost complex structure. 

A manifold endowed with an almost complex structure is said to be an almost complex manifold. If the structure is compatible with the metric the manifold is called an almost hermitian manifold. We can consider different types of submanifolds of an almost hermitian manifold (C. Ehresmann \cite{Eh2}). Two natural classes for instance are almost complex and totally real submanifolds which are defined as follows.  A submanifold $M$ of an almost complex manifold $\widetilde{M}$ is called almost complex (resp. totally real) if each tangent space of $M$ is mapped into itself (resp. into the normal space) by the almost complex structure of $\widetilde{M}\,.$ An almost complex submanifold is endowed with an almost complex structure induced by that of the ambient manifold.
C. Ehresmann showed (non-published results) that there exist always, for any dimension of the manifold, real submanifolds which is not always true for almost complex submanifolds if the dimension of the manifold is not 2. For example, $S^6$ does not admit almost complex submanifolds of dimension 4 and there does not exist almost complex functions on $S^6\,.$ The problem of existence of complex analytic functions on a complex manifold was also studied by S.S. Chern \cite{Ch}. 

We are interested in nearly K\"ahler manifolds, which are (pseudo-)Riemannian manifolds endowed with an almost complex structure compatible with a metric on the manifold satisfying one additional condition (Definition \ref{def-NK}). These manifolds have been studied intensively by Gray \cite{Gr}. In a recent work of J.-B. Butruille, \cite{Bu} it has been shown that the only homogeneous 6-dimensional Riemannian nearly K\"ahler manifolds are the nearly K\"ahler 6-sphere, $S^3\times S^3\,,$ the projective space $\mathbb{C}P^3$ and the flag manifold $SU(3)/U(1)\times U(1)\,.$

In this paper we study the analogue of the nearly K\"ahler space $S^3\times S^3$ in the pseudo-Riemannian case, which is $\mathrm{SL(2,\mathbb{R})}\times \mathrm{SL(2,\mathbb{R})}$. In particular, we study almost complex surfaces in the 6-dimensional space $\mathrm{SL(2,\mathbb{R})}\times \mathrm{SL(2,\mathbb{R})}$ where a nearly K\"ahler structure exists naturally. We give this structure explicitly and we investigate in details totally geodesic and parallel almost complex surfaces in this space.

\section{Preliminaries}
In this section, we will recall some basic definitions, properties and formulas that will be used in the following paper.
\\ An almost Hermitian manifold is a (pseudo-)Riemannian manifold $(M,g,J)$ which admits an endomorphism $J
$ of the tangent bundle such that
\begin{itemize}
\item $J^2=-Id$, i.e. $J$ is an almost complex structure,
\item $J$ is compatible with the metric $g$, i.e. $g(JX,JY)=g(X,Y)$ for all vector fields $X$ and $Y\,.$
\end{itemize}
We remark that the first condition requires the real dimension of $M$ to be even.
\begin{definition}\label{def-NK}
An almost Hermitian manifold is called a nearly K\"{a}hler manifold if the almost complex structure $J$ verifies $$(\widetilde{\nabla}_XJ)X=0$$ where, $\widetilde{\nabla}$ is the Levi-Civita connection associated to $g.$ This is equivalent to say that the tensor $G$ defined by $G(X,Y):=(\widetilde{\nabla}_XJ)Y$ is skew-symmetric.
\end{definition} \noindent For all vector fields $X,Y,Z\,,$ the tensor $G$ satisfies the following properties:
\begin{equation}\label{TensorG}
\begin{cases}
G(X,Y)+G(Y,X)&=0\,,\\
G(X,JY)+JG(X,Y)&=0\,,\\
g(G(X,Y),Z)+g(G(X,Z),Y)&=0\,.
\end{cases}
\end{equation}
\\
\noindent Remark that from the last two equations of \eqref{TensorG}, we can notice that the lowest dimension in which non-K\"{a}hler (i.e. $G$ does not vanish identically) nearly K\"{a}hler manifold can exist is 6.\\
\noindent There are two very natural types of submanifolds of nearly K\"{a}hler manifolds to be studied: almost complex submanifolds for which the  tangent spaces are invariant by the almost complex structure $J$ (i.e. $JTM=TM$) and totally real submanifolds for which $J$ interchanges tangent and normal vectors (i.e. $JTM\perp TM$).

\noindent Moreover, it will be convenient to mention some formulas that will be used through the calculation of the paper. First, the formula of Gauss which gives the decomposition of $\nabla_XY$ for tangent vectors $X,Y$ on a submanifold $M$, where $\nabla$ is the induced connection on $M.$ The Gauss formula is given by  $$\nabla_XY=\widetilde{\nabla}_XY+h(X,Y)$$ where, $\widetilde{\nabla}$ is the Levi-Civita connection on the ambient space and $h$ is the second fundamental form. Next, the formula of Weingarten that gives the decomposition along a tangent and a normal vector $\xi$, is given by $$\nabla_X\xi=-A_{\xi}X+\nabla^{\perp}_X\xi$$ where, $A$ is the shape operator given by the relation $g(h(X,Y),\xi)=g(A_{\xi}X,Y)$ and $\nabla^{\perp}$ is the normal connection on $M\,.$ \\
We recall from \cite{Bo-Di-Di-Vr} some useful relations that hold for an almost complex submanifold $M$ and that follow from the Gauss and Weingarten formulas:
\begin{align*}
\nabla_X JX&=J\nabla_X X, & h(X,JY)&=Jh(X,Y),\\
A_{J\xi}X=JA_{\xi} X&=-A_{\xi}JX & G(X,\xi)&=\nabla_X^{\perp}J\xi-J\nabla_X^{\perp}\xi
\end{align*}
where, $X,Y$ are tangent vectors and $\xi$ is a normal vector of $M.$ \\ \\
Finally, let us finish this section by the Gauss, Codazzi  equations  that are respectively given by:
\begin{equation*}
\begin{cases}
g(\widetilde{R}(X,Y)Z,W)=g(R(X,Y)Z,W)+g(h(X,Z),h(Y,W))-g(h(Y,Z),h(X,W))\\
(\widetilde{R}(X,Y)Z)^\perp =(\nabla\widetilde{h})(X,Y,Z)-(\nabla\widetilde{h})(Y,X,Z)
\end{cases}
\end{equation*}
Another usefull property is the Ricci identity which states that
\begin{equation*}
(\nabla^2h)(X,Y,Z,W)-(\nabla^2h)(Y,X,Z,W)=R^\perp(X,Y)h(Z,W)-h(R(X,Y)Z,W)-h(X,R(Y,Z)W).
\end{equation*}

\section{The nearly K\"{a}hler structure on SL(2,$\mathbb{R}$)$\times$SL(2,$\mathbb{R}$)}
To begin with, consider the nondegenerate indefinite inner product in $\mathbb{R}^4$ given by
\begin{equation}
\langle a,b\rangle=-\tfrac{1}{2}(a_1b_4-a_2b_3-a_3b_2+a_4b_1)\,,
\end{equation}
where $a=(a_1,a_2,a_3,a_4)$, $b=(b_1,b_2,b_3,b_4)\in\mathbb{R}^4\,.$

\noindent By identifying the $2\times 2$ real matrices space, $M(2,\mathbb{R})$, with the 4-dimensional Euclidean space $\mathbb{R}^4$, the above inner product can be viewed on $M(2,\mathbb{R})$ as
\begin{equation}
\left\langle A,B\right\rangle=-\tfrac{1}{2}{\rm Trace}(({\rm adj} A)^T B) \quad {\rm for}\ A,B\in M(2,\mathbb{R})\,.
\end{equation}
Therefore, the space of $2\times 2$-real matrices with a determinant $1$ denoted by $SL(2,\mathbb{R}),$ is given by
$$
SL(2,\mathbb{R})=\left\{A\in M(2,\mathbb{R}) | \left\langle A, A \right\rangle=-1\right\}.
$$ 

\noindent The restriction of $\left\langle.,.\right\rangle$ to the tangent spaces at the points of $SL(2,\mathbb{R})$ will be denoted also by $\left\langle.,.\right\rangle\,.$ We then consider $SL(2,\mathbb{R})$ equipped with the pseudo-Riemannian metric $\left\langle.,.\right\rangle$.

\noindent Now, let us proceed by defining the tangent space of $\mathbf{ SL(2,\mathbb{R})\times SL(2,\mathbb{R})}.$ For that aim, let $(A,B)\in SL(2,\mathbb{R})\times SL(2,\mathbb{R})$. By the natural identification $$T_{(A,B)}(SL(2,\mathbb{R})\times SL(2,\mathbb{R}))\cong T_ASL(2,\mathbb{R})\oplus T_BSL(2,\mathbb{R}),$$ we may write a tangent vector at $(A,B)$ as $Z(A,B)=(U(A,B), V(A,B))$ or simply $Z=(U,V)\,.$
\\
Let $A=\begin{pmatrix}
a & b\\
c & d
\end{pmatrix}\in SL(2,\mathbb{R})$. The tangent vector fields $X_1,X_2$ and $X_3,$  given by
\begin{eqnarray*}
X_1(A)&=&A\begin{pmatrix}
1 & 0\\
0 & -1
\end{pmatrix}=\begin{pmatrix}
a & -b\\
c & -d
\end{pmatrix},\\
X_2(A)&=&A\begin{pmatrix}
0 & 1\\
1 & 0
\end{pmatrix}=\begin{pmatrix}
b & a\\
d & c
\end{pmatrix},\\
X_3(A)&=&A\begin{pmatrix}
0 & 1\\
-1 & 0
\end{pmatrix}=\begin{pmatrix}
-b & a\\
-d & c
\end{pmatrix},
\end{eqnarray*}
form an orthogonal (semi-orthonormal) basis of $T_ASL(2,\mathbb{R})$ such that
$$
\left\langle X_1,X_1\right\rangle =1,\quad \left\langle X_2,X_2\right\rangle =1\quad {\rm and}\ \left\langle X_3,X_3\right\rangle =-1. 
$$
Hence, the tangent space of $SL(2,\mathbb{R})$ can be defined as 
$$
T_ASL(2,\mathbb{R})=\left\{A\alpha\ |\ \alpha\ {\rm is\ a\ } 2\times 2\ {\rm matrix\ of\ trace\ 0}\right\}.
$$
Consequently, the vector fields
\begin{align*}
E_1(A,B) &=\left(A\begin{pmatrix}
1 & 0\\
0 & -1
\end{pmatrix},0\right),& F_1(A,B)&=\left(0,B\begin{pmatrix}
1 & 0\\
0 & -1
\end{pmatrix}\right),\\
E_2(A,B) &=\left(A\begin{pmatrix}
0 & 1\\
1 & 0
\end{pmatrix},0\right), &F_2(A,B)&=\left(0,B\begin{pmatrix}
0 & 1\\
1 & 0
\end{pmatrix}\right),\\
E_3(A,B) &=\left(A\begin{pmatrix}
0 & 1\\
-1 & 0
\end{pmatrix},0\right),& F_3(A,B)&=\left(0,B\begin{pmatrix}
0 & 1\\
-1 & 0
\end{pmatrix}\right),
\end{align*}
are mutually orthogonal with respect to the usual product metric on $SL(2,\mathbb{R})\times SL(2,\mathbb{R})$ given by:
$$
\left\langle Z,Z'\right\rangle_{SL(2,\mathbb{R})\times SL(2,\mathbb{R})}=\left\langle U,U'\right\rangle+\left\langle V,V'\right\rangle
$$
for $Z=(U,V), Z'=(U',V')$ tangential to $SL(2,\mathbb{R})\times SL(2,\mathbb{R})\,.$
We note that the usual product metric $\left\langle .,.\right\rangle_{SL(2,\mathbb{R})\times SL(2,\mathbb{R})}$ will also be denoted by $\left\langle .,.\right\rangle\,.$
Hence, the lie brackets are
\begin{align*}
[E_1,E_2]&=2E_3\qquad [F_1,F_2]=2F_3\qquad [E_i,E_i]=0&\\
[E_1,E_3]&=2E_2\qquad [F_1,F_3]=2F_2\qquad [E_i,F_j]=0&\\
[E_2,E_3]&=-2E_3\ \quad [F_2,F_3]=-2F_1\ \quad [F_i,F_i]=0&\\
\end{align*}

\vspace{-0.7cm}

for $i,j\in\{1,2,3\}\,.$

\noindent Now, we define the almost complex structure $J$ on $SL(2,\mathbb{R})\times SL(2,\mathbb{R})$ by 
\begin{eqnarray*}
JE_i&=&\frac{1}{\sqrt{3}}\left(2F_i+E_i\right),\\
JF_i&=&-\frac{1}{\sqrt{3}}\left(E_i+F_i\right)\,,
\end{eqnarray*}
or more generally by
\begin{equation*}
J(A\alpha,B\beta)=\frac{1}{\sqrt{3}}\left(A(\alpha-2\beta), B(2\alpha-\beta)\right)\,,
\end{equation*}
for $\alpha,\beta$ 2-dimensional matrices of trace 0 and therefore, $\displaystyle{(A\alpha,B\beta)\in T_{(A,B)}(SL(2,\mathbb{R})\times SL(2,\mathbb{R}))}\,.$\\

\noindent The metric on $SL(2,\mathbb{R})\times SL(2,\mathbb{R})$ which corresponds to the almost complex structure $J$ is the metric $g$ given as follows:
\begin{align*}
g(E_i,F_j)=\left\{\begin{aligned}
-\frac{1}{3}\delta_{ij}&\quad {\rm for}&\ i&=&1,2,\\
\frac{1}{3}\delta_{ij} &\quad {\rm for}&\ i&=&3,
\end{aligned}
\right.
\end{align*}
\begin{align*}
g(E_i,E_j)=g(F_i,F_j)=\left\{\begin{aligned}
\frac{2}{3}\delta_{ij}&\quad {\rm for}&\ i&=&1,2,\\
-\frac{2}{3}\delta_{ij} &\quad {\rm for}&\ i&=&3\,.
\end{aligned}
\right.
\end{align*}
More generally, on arbitrary tangent vectors in $T_{(A,B)}(SL(2,\mathbb{R})\times SL(2,\mathbb{R}))$, the metric $g$ is given in function of the usual product metric on $SL(2,\mathbb{R})\times SL(2,\mathbb{R})$ by
\begin{equation}\label{g in function of <,> explicitly}
g((A\alpha,B\beta),(A\gamma,B\delta))=\frac{2}{3}\left\langle(A\alpha,B\beta),(A\gamma,B\delta)\right\rangle-\frac{1}{3}\left\langle(A\beta,B\alpha),(A\gamma,B\delta)\right\rangle
\end{equation}
for any $\alpha,\beta,\gamma,\delta\in M(2,\mathbb{R})$ of trace 0.

\noindent We can check that the almost complex structure $J$ is compatible with the metric $g\,.$\\

Let us now consider the following Lemma that will be used through the paper.
\begin{lemma}
\label{lemma L-C connetion and its derivative}
The Levi-Civita connection $\widetilde{\nabla}$ on $\mathrm{SL(2,\mathbb{R})}\times\mathrm{SL(2,\mathbb{R})}$ with respect to the metric $g$ is given by

\small{
\begin{align*}
\widetilde{\nabla}_{E_1}E_1&=0, & \widetilde{\nabla}_{E_2}E_1&=-E_3, & \widetilde{\nabla}_{E_3}E_1&=-E_2, \\
\widetilde{\nabla}_{E_1}E_2&=E_3,& \widetilde{\nabla}_{E_2}E_2&=0, & \widetilde{\nabla}_{E_3}E_2&=E_1, \\
\widetilde{\nabla}_{E_1}E_3&=E_2, & \widetilde{\nabla}_{E_2}E_3&=-E_1, & \widetilde{\nabla}_{E_3}E_3&=0,\\
\widetilde{\nabla}_{E_1}F_1&=0, & \widetilde{\nabla}_{E_2}F_1&=\tfrac{1}{3}(E_3-F_3), & \widetilde{\nabla}_{E_3}F_1&=\tfrac{1}{3}(E_2-F_2), \\
\widetilde{\nabla}_{E_1}F_2&=\tfrac{1}{3}(-E_3+F_3), & \widetilde{\nabla}_{E_2}F_2&=0,& \widetilde{\nabla}_{E_3}F_2&=\tfrac{1}{3}(-E_1+F_1), \\
\widetilde{\nabla}_{E_1}F_3&=\tfrac{1}{3}(-E_2+F_2), & \widetilde{\nabla}_{E_2}F_3&=\tfrac{1}{3}(E_1-F_1), & \widetilde{\nabla}_{E_3}F_3&=0,\\
\widetilde{\nabla}_{F_1}E_1&=0, & \widetilde{\nabla}_{F_2}E_1&=\tfrac{1}{3}(-E_3+F_3), & \widetilde{\nabla}_{F_3}E_1&=-\tfrac{1}{3}(E_2-F_2), \\
\widetilde{\nabla}_{F_1}E_2&=\tfrac{1}{3}(E_3-F_3), & \widetilde{\nabla}_{F_2}E_2&=0, & \widetilde{\nabla}_{F_3}E_2&=\tfrac{1}{3}(E_1-F_1), \\
\widetilde{\nabla}_{F_1}E_3&=\tfrac{1}{3}(E_2-F_2), & \widetilde{\nabla}_{F_2}E_3&=-\tfrac{1}{3}(E_1-F_1), & \widetilde{\nabla}_{F_3}E_3&=0,\\
\widetilde{\nabla}_{F_1}F_1&=0, & \widetilde{\nabla}_{F_2}F_1&=-F_3, & \widetilde{\nabla}_{F_3}F_1&=-F_2, \\
\widetilde{\nabla}_{F_1}F_2&=F_3, & \widetilde{\nabla}_{F_2}F_2&=0, & \widetilde{\nabla}_{F_3}F_2&=F_1, \\
\widetilde{\nabla}_{F_1}F_3&=F_2, & \widetilde{\nabla}_{F_2}F_3&=-F_1, & \widetilde{\nabla}_{F_3}F_3&=0\,.
\end{align*}
}
Then its covariant derivative, $\widetilde{\nabla}J$ can be computed as follows 
\small{
\begin{align*}
(\widetilde{\nabla}_{E_1}J)E_1&=0,& (\widetilde{\nabla}_{E_2}J)E_1&=\tfrac{2}{3\sqrt{3}}(E_3+2F_3), & (\widetilde{\nabla}_{E_3}J)E_1&=\tfrac{2}{3\sqrt{3}}(E_2+2F_2), \\
(\widetilde{\nabla}_{E_1}J)E_2&=-\tfrac{2}{3\sqrt{3}}(E_3+2F_3), & (\widetilde{\nabla}_{E_2}J)E_2&=0, & (\widetilde{\nabla}_{E_3}J)E_2&=-\tfrac{2}{3\sqrt{3}}(E_1+2F_1), \\
(\widetilde{\nabla}_{E_1}J)E_3&=-\tfrac{2}{3\sqrt{3}}(E_2+2F_2), & (\widetilde{\nabla}_{E_2}J)E_3&=\tfrac{2}{3\sqrt{3}}(E_1+2F_1), & (\widetilde{\nabla}_{E_3}J)E_3&=0,\\
(\widetilde{\nabla}_{E_1}J)F_1&=, & (\widetilde{\nabla}_{E_2}J)F_1&=\tfrac{2}{3\sqrt{3}}(E_3-F_3), & (\widetilde{\nabla}_{E_3}J)F_1&=\tfrac{2}{3\sqrt{3}}(E_2-F_2), \\
(\widetilde{\nabla}_{E_1}J)F_2&=\tfrac{2}{3\sqrt{3}}(-E_3+F_3),  & (\widetilde{\nabla}_{E_2}J)F_2&=0, & (\widetilde{\nabla}_{E_3}J)F_2&=\tfrac{2}{3\sqrt{3}}(-E_1+F_1), \\
(\widetilde{\nabla}_{E_1}J)F_3&=\tfrac{2}{3\sqrt{3}}(-E_2+F_2), & (\widetilde{\nabla}_{E_2}J)F_3&=\tfrac{2}{3\sqrt{3}}(E_1-F_1), & (\widetilde{\nabla}_{E_3}J)F_3&=0,\\
(\widetilde{\nabla}_{F_1}J)E_1&=0, & (\widetilde{\nabla}_{F_2}J)E_1&=\tfrac{2}{3\sqrt{3}}(E_3-F_3), & (\widetilde{\nabla}_{F_3}J)E_1&=\tfrac{2}{3\sqrt{3}}(E_2-F_2), \\
(\widetilde{\nabla}_{F_1}J)E_2&=-\tfrac{2}{3\sqrt{3}}(E_3-F_3), & (\widetilde{\nabla}_{F_2}J)E_2&=0, & (\widetilde{\nabla}_{F_3}J)E_2&=-\tfrac{2}{3\sqrt{3}}(E_1-F_1), \\
(\widetilde{\nabla}_{F_1}J)E_3&=-\tfrac{2}{3\sqrt{3}}(E_2-F_2), & (\widetilde{\nabla}_{F_2}J)E_3&=\tfrac{2}{3\sqrt{3}}(E_1-F_1), & (\widetilde{\nabla}_{F_3}J)E_3&=0, \\
(\widetilde{\nabla}_{F_1}J)F_1&=0, & (\widetilde{\nabla}_{F_2}J)F_1&=-\tfrac{2}{3\sqrt{3}}(2E_3+F_3), & (\widetilde{\nabla}_{F_3}J)F_1&=-\tfrac{2}{3\sqrt{3}}(2E_2+F_2), \\
(\widetilde{\nabla}_{F_1}J)F_2&=\tfrac{2}{3\sqrt{3}}(2E_3+F_3), & (\widetilde{\nabla}_{F_2}J)F_2&=0 ,& (\widetilde{\nabla}_{F_3}J)F_2&=\tfrac{2}{3\sqrt{3}}(2E_1+F_1), \\
(\widetilde{\nabla}_{F_1}J)F_3&=\tfrac{2}{3\sqrt{3}}(2E_2+F_2), & (\widetilde{\nabla}_{F_2}J)F_3&=-\tfrac{2}{3\sqrt{3}}(2E_1+F_1), & (\widetilde{\nabla}_{F_3}J)F_3&=0\,.
\end{align*}
}
\end{lemma}
\noindent Let us set $G:=\widetilde{\nabla}$. Then, $G$ is skew symmetric and satisfies the following equations
\begin{equation*}
G(X,JY)+JG(X,Y)=0\,, \quad g(G(X,Y),Z)+g(G(X,Z),Y)=0\,,
\end{equation*}
for any vector fields $X,Y,Z\in T(SL(2,\mathbb{R})\times SL(2,\mathbb{R}))\,.$ Therefore, $SL(2,\mathbb{R})\times SL(2,\mathbb{R})$ endowed with $g$ and $J$, becomes a nearly K\"{a}hler  manifold.

\noindent In order to express the curvature tensor of the nearly K\"{a}hler $SL(2,\mathbb{R})\times SL(2,\mathbb{R})$, it is convenient to introduce an almost product structure $P$ on which is defined as
\begin{eqnarray}
P(pU,qV)=(pV,qU),\quad \forall Z=(pU,qV)\in T_{(p,q)}(SL(2,\mathbb{R})\times SL(2,\mathbb{R})).
\end{eqnarray}
We can verify easily that $P$ satisfies the following properties:
\begin{align*}
P^2 &=Id\,,\\
PJ&=-JP\,,\\
g(PZ,PZ')&=g(Z,Z'),\ {\rm i.e.}\ P\  {\rm is\ compatible\ with\ } g\,,\\
g(PZ,Z')&=g(Z,PZ'),\ {\rm i.e.}\ P\ {\rm is\ symmetric\,.}  
\end{align*}
It then turns out that the Riemann curvature tensor $\widetilde{R}$ on $(SL(2,\mathbb{R}))\times SL(2,\mathbb{R}),g)$ is given by
\begin{eqnarray*}
\widetilde{R}(U,V)W&=&-\tfrac{5}{6}\left(g(V,W)U-g(U,W)V\right)\\\notag
 & &-\tfrac{1}{6}\left(g(JV,W)JU-g(JU,W)JV-2g(JU,V)JW\right)\\
 & &-\tfrac{2}{3}\left(g(PV,W)PU-g(PU,W)PV+g(JPV,W)JPU-g(JPU,W)JPV\right)\,,
\end{eqnarray*}
and the tensors $\widetilde{\nabla}G$ and $G$ satisfy
\begin{eqnarray}
(\widetilde{\nabla}G)(X,Y,Z)&=&-\tfrac{2}{3}\left(g(X,Z)JY-g(X,Y)JZ-g(JY,Z)X\right)\,,\\ 
g(G(X,Y),G(Z,W))&=&-\tfrac{2}{3}\left(g(X,Z)g(Y,W)-g(X,W)g(Y,Z)\right.\label{g(G,G)}\\\notag
& &+\left. g(JX,Z)g(JW,Y)-g(JX,W)g(JZ,Y)\right)
\end{eqnarray}
From \eqref{g(G,G)} and by using the fact that $g(G(X,Y),Z)=-g(G(X,Z),Y)$ we deduce the following equation
\begin{eqnarray}\label{G(X,G(Z,W))}
G(X,G(Z,W))&=&\tfrac{2}{3}\left(g(X,Z)W-g(X,W)Z+g(JX,Z)JW\right.\\\notag
& &\quad\left. -g(JX,W)JZ\right).
\end{eqnarray}

\noindent Moreover, we can express the tensor $G$ explicitly for any tangent vectors fields. For that aim, let us present the following proposition.
\begin{proposition}\label{proposition G explicity}
Let $X=(A\alpha,B\beta), Y=(A\gamma,B\delta)\in T_{(A,B)}(SL(2,\mathbb{R})\times SL(2,\mathbb{R}))\,.$ Then,
$$
G(X,Y)=\tfrac{2}{3\sqrt{3}}\left(A(-\alpha\times\gamma-\alpha\times\delta+\gamma\times\beta+2\beta\times\delta),B(-2\alpha\times\gamma+\alpha\times\delta-\gamma\times\beta+\beta\times\delta)\right)
$$
where $\times$ is a product similar to the product vector, that we define on the space of real matrices of dimension $2$ and  trace $0$ by $\alpha\times\beta=\tfrac{1}{2}(\alpha\beta-\beta\alpha)\,.$
\end{proposition}
\begin{proof}
Let $\alpha_1,\alpha_2,\alpha_3$ be the coefficients of $\alpha$ in the basis {\footnotesize $\left\{\begin{pmatrix}
1& 0\\
0& -1
\end{pmatrix}, \begin{pmatrix}
0&1\\
1&0
\end{pmatrix}, \begin{pmatrix}
0&1\\
-1&0
\end{pmatrix}\right\}
$}, similarly for $\beta\,.$ Then, we write $$X=U_{\alpha}+V_\beta$$ where, $U_{\alpha}=\alpha_1 E_1+\alpha_2 E_2+\alpha_3 E_3$ and $V_{\beta}=\beta_1F_1+\beta_2 F_2+\beta_3 F_3$. Similarly, $$Y=U_{\gamma}+V_{\delta}\,.$$ By using Lemma \ref{lemma L-C connetion and its derivative}, we can compute 
$$
G(U_{\alpha},V_{\beta})=-\tfrac{2}{3\sqrt{3}}(U_{\alpha\times\beta}-V_{\alpha\times\beta}),\quad 
G(U_{\alpha},U_{\beta})=-\tfrac{2}{3\sqrt{3}}(U_{\alpha\times\beta}+2V_{\alpha\times\beta}).
$$
As $PU_{\alpha}=V_{\alpha},$ we obtain $$G(V_{\alpha},V_{\beta})=\tfrac{2}{3\sqrt{3}}(V_{\alpha\times\beta}+2U_{\alpha\times\beta})\,.$$
Then, by linearity we complete the proof of the proposition.
\end{proof}
Moreover, the almost product structure $P$ can be expressed in terms of the usual product structure $Q$ given by $QZ=Q(A\alpha,B\beta)=(-A\alpha,B\beta)$ and vice versa:
\begin{align}
QZ&=-\frac{1}{\sqrt{3}}(2PJZ-JZ)\,,\label{Q}\\
PZ&=-\frac{1}{2}(Z+\sqrt{3}JQZ)\,.
\end{align}
The metric $g$ is expressed in terms of the metric $\langle.,.\rangle$ by:
$$
g(Z,Z')=\frac{1}{4}\left(\langle Z,Z'\rangle +\langle JZ,JZ'\rangle\right).
$$
Then,
$$
g(QZ,QZ')+g(Z,Z')=\frac{4}{3}\langle Z,Z'\rangle\,,
$$
so that the metric $\langle.,.\rangle$ can be written in terms of the metric $g$:
\begin{equation}\label{<,> in function of g}
\langle Z,Z'\rangle=2g(Z,Z')+g(Z,PZ').
\end{equation}
Let us end this section with the following.
\begin{lemma}\label{lemma relation between connections}
The relation between the Levi-Civita connection $\widetilde{\nabla}$ of the metric $g$ and that of the usual product metric $\langle.,.\rangle$, denoted $\nabla^E$, is

$$
\nabla^E_XY=\widetilde{\nabla}_XY+\frac{1}{2}\left(JG(X,PY)+JG(Y,PX)\right).
$$
\end{lemma}
\section{Almost complex surfaces in SL(2,$\mathbb{R}$)$\times$SL(2,$\mathbb{R}$)}

In this section, we start by studying almost complex surfaces in $SL(2,\mathbb{R})\times SL(2,\mathbb{R})$ which are totally geodesic. The reasoning in some cases is similar to that applied in \cite{Bo-Di-Di-Vr} in the case of the nearly K\"{a}hler $S^3\times S^3\,.$ We note that the identities in Lemma 3.1. of \cite{Bo-Di-Di-Vr} remain true in our case and we will always assume that the almost complex surface is a regular surface, i.e. the induced metric is non degenerate. As $J$ is compatible with the metric this either implies that the induced metric is positive definite or negative definite. 
\begin{proposition}\label{Proposition totally geodesic}
If $M$ is a totally geodesic almost complex surface in SL(2,$\mathbb{R}$)$\times$SL(2,$\mathbb{R}$), then either
\begin{itemize}
\item[(1)] P maps the tangent space into the normal space and the Gaussian curvature $K$ is $-\tfrac{4}{3}$
\item[(2)] P preserves the tangent space (and therefore also the normal space) and the Gaussian curvature is $0$.
\end{itemize}
\end{proposition}
 \begin{proof}
 Let $(A,B)\in M$ be a point of $M$ and $v$ a tangent vector to $M$ at $(A,B)$ such that $g(v,v)= \pm 1$.\\
\noindent {\bf Case 1:} If $g(v,v)=1$. Using Codazzi's equation, we have $(\widetilde{R}(v,Jv)v)^\perp=0$ then,  $\widetilde{R}(v,Jv)v$ is a multiple of $Jv$. By the Gauss equation, and using the fact that $h=0$, we have 
 \begin{equation*}
 R(v,Jv)v=-\tfrac{4}{3}\left(-Jv+g(PJv,v)Pv-g(Pv,v)PJv\right).
 \end{equation*}
 Since the metric $g$ is positive definite on $v$, we can choose $v$ such that $g(v,Pv)$ is maximal for all unit vectors in $(A,B)$, so that $g(Pv,Jv)=g(PJv,v)=0$ (since $P$ is symmetric). Then the Gauss equation becomes
 \begin{equation*}
 R(v,Jv)v=\tfrac{4}{3}\left(Jv+g(Pv,v)PJv\right).
\end{equation*}
Now, we will distinguish between two cases.\\
\noindent If $g(Pv,v)=0$, then $K=g(R(v,Jv)Jv,v)=-\frac{4}{3}.$ In this case $Pv$ and $PJv$ are normal vectors because $g(Pv,v)=g(Pv,Jv)=0$ and $g(PJv,v)=0$, $g(PJv,Jv)=-g(JPv,Jv)=-g(Pv,v)=0$. We can verify in this case that the vectors $v, Jv, Pv$ and $PJv$ are of length $1$ and $G(v,Pv)$, $JG(v,Pv)$ are of length $-2/3\,.$\\

\noindent If $g(Pv,v)\neq0$, we have $g(PJv,Jv)=-g(Pv,v)\neq 0$ which implies that $\displaystyle{PJv=-g(Pv,v)Jv}$, a non-zero multiple of $Jv$ then, $g(PJv,PJv)=-g(Pv,v)g(Jv,PJv).$ Therefore $\displaystyle{g(Pv,v)=\pm1}$ and $PJv=\pm Jv$. We assume that $PJv=-Jv$ then, $Pv=v$ and $K=-\tfrac{4}{3}g(Jv+g(Pv,v)PJv,Jv)=-\tfrac{4}{3}(1-g(Pv,v)^2)=0\,.$\\

\noindent {\bf Case 2:} If $g(v,v)=-1$. Using Codazzi's equation we have $\widetilde{R}(v,Jv)v$ is a multiple of $Jv$ and by Gauss equation, we deduce
$$
R(v,Jv)v=-\tfrac{4}{3}\left(Jv+g(PJv,v)Pv-g(Pv,v)PJv\right)\,.
$$
We can suppose $g(PJv,v)=0$ then, the Gauss equation simplifies to
$$
R(v,Jv)v=-\tfrac{4}{3}\left(Jv-g(Pv,v)PJv\right)\,.
$$
In this case, the assumption $g(Pv,v)=0$ is not possible since the vectors $Pv$ and $PJv$ are normal and so we have two tangent vectors ($v$ and $Jv$) and two normal vectors $Pv$ and $PJv$ of negative length. Hence $g(Pv,v)\neq0$, as before, $g(Pv,v)=\pm1$, and consequently, $PJv=\pm Jv$. We assume that $PJv=-Jv$ which implies $Pv=v$ and $K=0$.
 \end{proof}

Let us now investigate in more detail the two cases introduced in the previous proposition. We start with a flat surface for which $P$ preserves both the tangent and the normal space.  

\vspace{0.5cm}

\subsection*{Explicit examples}
In this section, we find explicitly the submanifolds $M$ such that $PTM\subset TM$ with Gauss curvature $K=0$: 

\vspace{0.3cm}

\noindent Let $(A,B)\in M$ be a point of $M$ and $v$ a tangent vector to $M$ such that $g(v,v)=1\,.$ The vectors $v$ and $Jv$ form an orthonormal basis of $T_{(A,B)}M\,.$ Also, we have
\begin{equation}\label{Pv=v}
Pv=v
\end{equation}
and
\begin{equation}\label{PJv=-Jv}
 PJv=-Jv.
\end{equation}
Similar to Lemma 3.1. in \cite{Bo-Di-Di-Vr}, we know in this case  that $(\nabla_XP)Y=0$ for any tangent vectors $X,Y$ which leads to $P(\nabla_Xv)=\nabla_Xv$. On the other hand, $\nabla_Xv$ is a multiple of $Jv$ which gives $P(\nabla_Xv)=-\nabla_Xv$. Then for any tangent vector $X\,,$
$$
\nabla_Xv=0\,,
$$
similarly for 
$$
\nabla_XJv=0\,.
$$

Consequently, $[v,Jv]=0$ and therefore we can find coordinates $s$ and $t$ on the surface such that $$v:=F_s\qquad {\rm and}\qquad Jv:=F_t,$$
where $F$ denotes the immersion, i.e.  $$F:M\to SL(2,\mathbb{R})\times SL(2,\mathbb{R}):(s,t)\mapsto (A(s,t),B(s,t)).$$  

\noindent Similar to \cite{Bo-Di-Di-Vr}, we may assume that $F_t=JF_s$ and therefore there are $2\times 2$-real matrices $\alpha, \beta, \gamma, \delta$ with vanishing trace such that
$$
A_s=A\alpha,\quad B_s=B\beta,\quad A_t=A\gamma,\quad B_t=B\delta.
$$
The matrices $\alpha, \beta, \gamma, \delta$ are such that
$$
\alpha=\beta\,,\quad \gamma=-\delta=-\frac{\alpha}{\sqrt{3}}\,. 
$$
This comes from $PF_s=Fs, JF_s=F_t$, 
which implies the following equations
\begin{align}\label{first-derivatives-case1}
\left\{\begin{aligned}
A_s&=A\alpha\,\
A_t&=-\frac{1}{\sqrt{3}}A\alpha,
\end{aligned}
\right.
\quad {\rm and}\quad
\left\{\begin{aligned}
B_s&=B\alpha,\\
B_t&=\frac{1}{\sqrt{3}}B\alpha.
\end{aligned}
\right.
\end{align} Now, we will distinguish between two cases:
\begin{itemize}
\item If $g$ is positive definite:
From \eqref{g in function of <,> explicitly}, we have 
$$
\langle \alpha,\alpha\rangle=\langle\beta,\beta\rangle=\frac{3}{2}\quad {\rm and}\quad \langle\gamma,\gamma\rangle=\langle\delta,\delta\rangle=\frac{1}{2}.
$$ 
Moreover, by \eqref{<,> in function of g}, we have
\begin{equation*}
\langle F_s, F_s\rangle=3\,,\quad \langle F_s,F_t\rangle=0\,,\quad\langle F_t,F_t\rangle=1
\end{equation*}
and using \eqref{Q}, we have in general $\langle X,QY\rangle=-\sqrt{3}g(X,PJY)$ so that
\begin{equation*}
\langle F_s,QF_s\rangle=\langle F_t, QF_t\rangle=0\,,\quad \langle F_s,QF_t\rangle=\langle F_t, Q F_t\rangle=\sqrt{3}\,.
\end{equation*}
Also, the usual connection $\nabla$ vanishes on the vectors $F_s$ and $F_t\,,$ then we deduce
\begin{equation*}
F_{ss}=\frac{3}{2} F\,,\quad F_{tt}=\frac{1}{2} F\quad {\rm and}\quad F_{st}=\frac{\sqrt{3}}{2}QF\,,
\end{equation*}
which is equivalent to say that the second derivatives of the components $A$ and $B$ satisfy the following equations
\begin{align}\label{second-derivatives-case1}
\left\{\begin{aligned}
A_{ss}&=\frac{3}{2}A\,,\\
A_{st}&=-\frac{\sqrt{3}}{2}A\,,\\
A_{tt}&=\frac{1}{2}A\,,
\end{aligned}
\right.
\quad {\rm and}\quad
\left\{\begin{aligned}
B_{ss}&=\frac{3}{2}B\,,\\
B_{st}&=\frac{\sqrt{3}}{2}B\,,\\
B_{tt}&=\frac{1}{2}B\,.\\
\end{aligned}
\right.
\end{align}
On the other hand, the integrablity condition $F_{st}=F_{ts}$ implies that $\alpha$ is a constant matrix. Hence, by deriving \eqref{first-derivatives-case1}, we get
\begin{align}\label{derivative of the 1st system}
\left\{\begin{aligned}
A_{ss}&=A\alpha\alpha,\\
A_{tt}&=\frac{1}{3}A\alpha\alpha,
\end{aligned}
\right.
\quad {\rm and}\quad
\left\{\begin{aligned}
B_{ss}&=B\alpha\alpha,\\
B_{tt}&=\frac{1}{3}B\alpha\alpha\,.
\end{aligned}
\right.
\end{align}
By identifying \eqref{second-derivatives-case1} and \eqref{derivative of the 1st system}, we obtain
$$
\alpha\alpha=\frac{3}{2}I\,,
$$
where $I$ is the identity matrix. The above equation is equivalent to
$$
\begin{pmatrix}
\alpha_1^2+\alpha_2^2-\alpha_3^2 & 0\\
0 & \alpha_1^2+\alpha_2^2-\alpha_3^2
\end{pmatrix}
=\begin{pmatrix}
\frac{3}{2} & 0\\
0 & \frac{3}{2}
\end{pmatrix}\,,
$$
where $\alpha_1, \alpha_2, \alpha_3$ are the coefficients of $\alpha$ in the basis
\small{$$\left\{\begin{pmatrix}
1 & 0\\
0 & -1
\end{pmatrix}\,,
\begin{pmatrix}
1 & 0\\
0 & -1
\end{pmatrix}\,,
\begin{pmatrix}
0 & 1\\
-1 & 0
\end{pmatrix}
\right\}.
$$}
\\
By using now that on the nearly K\"{a}hler $SL(2,\mathbb{R})\times SL(2,\mathbb{R})$, the map $$(p,q)\mapsto (ApC,BqC)$$ is an isometry. We can assume that $\alpha=\begin{pmatrix}
\sqrt{\frac{3}{2}} & 0\\
0 & -\sqrt{\frac{3}{2}}
\end{pmatrix}\,.
$

Now we are able to solve equations \eqref{first-derivatives-case1}.
\begin{proposition}
A flat totally geodesic  almost complex surface $M$  of $SL(2,\mathbb{R})\times SL(2,\mathbb{R})$ with positive definite induced metric is isometric to the  immersion $(s,t)\mapsto (A(s,t),B(s,t))$ where
\begin{equation*}
A(s,t)=\begin{pmatrix}
e^{\frac{1}{\sqrt{2}}\left(\sqrt{3}s-t\right)} & 0\\
0 & e^{-\frac{1}{\sqrt{2}}\left(\sqrt{3}s-t\right) }
\end{pmatrix}
\end{equation*}
and
\begin{equation*}
B(s,t)=\begin{pmatrix}
e^{\frac{1}{\sqrt{2}}\left(\sqrt{3}s+t\right)} & 0\\
0 & e^{-\frac{1}{\sqrt{2}}\left(\sqrt{3}s+t\right) }
\end{pmatrix}
\end{equation*}
\end{proposition}

\begin{proof}
To determine the solution $A$ of the first system of equations in \eqref{first-derivatives-case1}, we may pose $$x=\frac{\sqrt{3}s+t}{2}\quad{\rm and}\quad y=\frac{\sqrt{3}s-t}{2}\,.$$
The system is now equivalent to 
\begin{align*}
\left\{\begin{aligned}
A_x&=0\\
A_y&=\frac{2}{\sqrt{3}}A\alpha
\end{aligned}
\right.
\end{align*}
which has the solution
$$
A(x,y)=k\begin{pmatrix}
e^{\frac{2}{\sqrt{2}}y} & 0\\
0 & e^{-\frac{2}{\sqrt{2}}y}
\end{pmatrix}
$$
where $k$ is a constant. By applying some isometries, we can suppose that $k=\begin{pmatrix}
1 & 0\\
0 & 1
\end{pmatrix}\,.
$
Then we have the solution given in the proposition. In a similar way we obtain the solution $B\,.$
\end{proof}

\item We now study the case where $g$ is negative definite

In this case, $$\langle\alpha,\alpha\rangle=\langle\beta,\beta\rangle=-\frac{3}{2}\quad{\rm and}\quad \langle\gamma,\gamma\rangle=\langle\delta,\delta\rangle=-\frac{1}{2}\,,$$
and from \eqref{<,> in function of g},
\begin{equation*}
\langle F_s, F_s\rangle=-3\,,\quad \langle F_s,F_t\rangle=0\,,\quad\langle F_t,F_t\rangle=-1\,,
\end{equation*}
also from \eqref{Q},
\begin{equation*}
\langle F_s,QF_s\rangle=\langle F_t, QF_t\rangle=0\,,\quad \langle F_s,QF_t\rangle=\langle F_t, Q F_t\rangle=-\sqrt{3}\,.
\end{equation*}
We proceed similarly as before to obtain that the second derivatives of the components $A$ and $B$ satisfy in this case
\begin{align}\label{second-derivatives-case2}
\left\{\begin{aligned}
A_{ss}&=-\frac{3}{2}A\,,\\
A_{st}&=\frac{\sqrt{3}}{2}A\,,\\
A_{tt}&=-\frac{1}{2}A\,,
\end{aligned}
\right.
\quad {\rm and}\quad
\left\{\begin{aligned}
B_{ss}&=-\frac{3}{2}B\,,\\
B_{st}&=-\frac{\sqrt{3}}{2}B\,,\\
B_{tt}&=-\frac{1}{2}B\,.\\
\end{aligned}
\right.
\end{align}
and $\alpha$ is a constant matrix which satisfies 
$$
\langle\alpha,\alpha\rangle=-\frac{3}{2}
$$
hence, $ \alpha=\begin{pmatrix}
0 & \sqrt{\frac{3}{2}}\\
-\sqrt{\frac{3}{2}} & 0
\end{pmatrix}
$

Now the equations \eqref{first-derivatives-case1} can be solved by posing, as before, $x=\frac{\sqrt{3}s+t}{2}$ and $y=\frac{\sqrt{3}s-t}{2}$, we will obtain the solution given by the following.
\begin{proposition}
A flat totally geodesic  almost complex surface $M$  of $SL(2,\mathbb{R})\times SL(2,\mathbb{R})$ with negative definite induced metric is isometric to the  immersion $(s,t)\mapsto (A(s,t),B(s,t))$ where
\begin{equation*}
A(s,t)=\begin{pmatrix}
\cos\left(\frac{1}{\sqrt{2}}\left(\sqrt{3}s-t\right)\right) & \sin\left(\frac{1}{\sqrt{2}}\left(\sqrt{3}s-t\right)\right)\\
-\sin\left(\frac{1}{\sqrt{2}}\left(\sqrt{3}s-t\right)\right) & \cos\left(\frac{1}{\sqrt{2}}\left(\sqrt{3}s-t\right)\right) 
\end{pmatrix}
\end{equation*}
and
\begin{equation*}
B(s,t)=\begin{pmatrix}
\cos\left(\frac{1}{\sqrt{2}}\left(\sqrt{3}s+t\right)\right) & \sin\left(\frac{1}{\sqrt{2}}\left(\sqrt{3}s+t\right)\right)\\
-\sin\left(\frac{1}{\sqrt{2}}\left(\sqrt{3}s+t\right)\right) & \cos\left(\frac{1}{\sqrt{2}}\left(\sqrt{3}s+t\right)\right)
\end{pmatrix}
\end{equation*}
\end{proposition}
\end{itemize}

Now we start the study of almost complex surfaces with parallel second fundamental form. As mentioned before the properties of Lemma 3.1 of \cite{Bo-Di-Di-Vr} remain valid for almost complex surfaces in SL(2,$\mathbb{R}$)$\times$SL(2,$\mathbb{R})$. 

\begin{proposition}\label{prop-parallel-second-fundamental-form}
If $M$ is an almost complex surface in SL(2,$\mathbb{R}$)$\times$SL(2,$\mathbb{R}$) with parallel second fundamental form, then either
\begin{itemize}
\item[(1)] P maps the tangent space into the normal space and in this case, either $M$ is totally geodesic with constant Gaussian curvature $-\tfrac{4}{3}$ or $M$ has constant Gaussian curvature $-\tfrac{5}{9}$
\item[(2)] P preserves the tangent space, in this case the Gaussian curvature is $0$ and $M$ is totally geodesic and therefore $M$ is congruent to one of the two previous mentioned examples. 
\end{itemize}
\end{proposition}
\begin{proof}
Let $(A,B)\in M$ be a point of $M$ and $v$ a tangent vector to $M$ at $(A,B)$ such that $g(v,v)=\pm 1\,.$
\begin{itemize}
\item[-]{\bf Case 1} We suppose $g(v,v)=1\,.$\\
From Codazzi's equation, we can see that $\widetilde{R}(v,Jv)v$ is a multiple of $Jv\,.$ As we mentioned before, (in Proposition 4.1.) we can suppose $$g(Pv,Jv)=g(PJv,v)=0\,.$$ Then from Gauss equation, we have
$$
R(v,Jv)v=\widetilde{R}(v,Jv)v+2JA_{h(v,v)}v\,,
$$
with
$$
\widetilde{R}(v,Jv)v=\tfrac{4}{3}\left(Jv+g(Pv,v)PJv\right)\,.
$$
We study the two cases $g(Pv,v)=0$ and $g(Pv,v)\neq 0\,.$

\noindent If $g(Pv,v)\neq 0\,,$ then it follows from the Codazzi equation that $PJv$ is a tangent vector. Hence $P$ preserves the tangent space. Therefore we can take $v$ such that $g(Pv,v)=\pm 1\,,$ hence $PJV=\mp Jv\,.$ It follows then from the Gauss equation that $A_{h(v,v)} v=-\tfrac 12 K v$  and the Gauss curvature $K=-2||h(v,v)||^2\,.$ So, $||A_{h(v,v)}v||^2=||h(v,v)||^4=\frac{K^2}{4}\,.$ On the other hand, as $P$ preserves the tangent space we have that $\nabla P = 0$. Hence we deduce that the surface is flat and totally geodesic. 

\noindent If $g(Pv,v)=0\,,$ then $P$ maps tangent vectors into normal vectors and from the Gauss equation
$$
K=-\frac{4}{3}-2||h(v,v)||^2\,.
$$
Since $g(PJv,h(v,v))=g(Pv,h(v,v))=0\,,$ Ricci equation implies
\begin{eqnarray*}
g(R^{\perp}(v,Jv)h(v,v),Jh(v,v))&=&\frac{1}{3}||h(v,v)||^2-2||h(v,v)||^4\\
&=&-\frac{K^2}{2}-\frac{3}{2}K-\frac{10}{9}\,.
\end{eqnarray*}
On the other hand, Ricci identity gives
\begin{eqnarray*}
g(R^{\perp}(v,Jv)h(v,v),Jh(v,v))&=&2g(h(R(v,Jv)v,v),Jh(v,v))\\
&=&-2g(h(KJv,v),Jh(v,v))\\
&=& -2K||h(v,v)||^2\\
&=& K^2+\frac{4}{3}K\,.
\end{eqnarray*}
By identification, we have the equation
$$
\frac{3}{2}K^2+\frac{17}{6}K+\frac{10}{9}=0\,,
$$
which has the solutions $$K=-\frac{4}{3} \qquad K=-\frac{5}{9}\,.$$
\item[-]{\bf Case 2} We suppose $g(v,v)=-1\,.$
In a similar way as we did in case 1, we have
$$
R(v,Jv)v=-\frac{4}{3}\left(Jv-g(Pv,v)PJv\right)+2JA_{h(v,v)}v\,.
$$
If $g(Pv,v)\neq 0\,,$, similarly as before we get that $P$ preserves the tangent space.  Moreover, $g(Pv,v)PJv=-Jv$ and
$$
||h(v,v)||^2=-\frac{K}{2}\,.
$$
As $\nabla P=0$ and the eigenvalues of $P$ are constant, namely $\pm 1$, it again follows immediately that the surface is flat and therefore totally geodesic. 

\noindent If $g(Pv,v)=0\,,$  it follows that $G(v,Pv)$ is a normal with negative length. As the tangent space is negative definite in view of the index of the metric, we have that the normal space has to be positive definite. Hence we obtain a contradiction. 
\end{itemize}
\end{proof}

Now in order to be able to exclude the case with constant constant Gaussian  curvature $-\tfrac{5}{9}$ in the previous theorem and to get an explicit expression of the (totally geodesic) almost complex surface with constant Gaussian curvature $-\tfrac{4}{3}$ we will study more generally almost complex submanifolds with arbitrary Gauss curvature in more detail. We will in particular focus on those for which $P$ maps the tangent space into the normal space. 

 Let $M$ be an almost complex surface of ${\rm SL}(2,\mathbb{R})\times {\rm SL}(2,\mathbb{R})$ defined by the almost complex immersion $F:M\to {\rm SL}(2,\mathbb{R})\times {\rm SL}(2,\mathbb{R}): (s,t)\mapsto (A(s,t),B(s,t))$ where $(s,t)$ are the isothermal coordinates on $M\,.$ Similar as in  \cite{Bo-Di-Di-Vr},  we may assume $F_t=JF_s$ and there exist $2\times 2$-real matrices with vanishing trace $\widetilde{\alpha}, \widetilde{\beta}, \widetilde{\gamma}, \widetilde{\delta}$ such that 
\begin{equation}\label{first derivative}
A_s=A\widetilde{\alpha}, \quad A_t=A\widetilde{\beta},\quad  B_s=B\widetilde{\gamma},\quad B_t=B\widetilde{\delta}\,.
\end{equation}
Then the matrices $\widetilde{\alpha}, \widetilde{\beta}, \widetilde{\gamma}, \widetilde{\delta}$ are such that

\begin{eqnarray}
\widetilde{\gamma}&=&\frac{\widetilde{\alpha}}{2}-\frac{\sqrt{3}}{2}\widetilde{\beta}\label{gamma tilde and delta tilde in function of alpha tilde and beta tilde}\\
\widetilde{\delta}&=&\frac{\sqrt{3}}{2}\widetilde{\alpha}+\frac{\widetilde{\beta}}{2}
\end{eqnarray}
Furthermore, by using the integrability conditions  $A_{st}=A_{ts}$ and $B_{st}=B_{ts}$ we have the two equations
\begin{eqnarray}
{\widetilde{\alpha}}_t-{\widetilde{\beta}}_s&=&2\widetilde{\alpha}\times\widetilde{\beta}\label{integrability1}\\
{\widetilde{\alpha}}_s+{\widetilde{\beta}}_t&=&-\frac{2}{\sqrt{3}}\widetilde{\alpha}\times\widetilde{\beta}\label{integrability2}
\end{eqnarray}
where we define the product $\times$ by $\widetilde{\alpha}\times\widetilde{\beta}=\tfrac{1}{2}(\widetilde{\alpha}\widetilde{\beta}-\widetilde{\beta}\widetilde{\alpha})\,.$

Now by writing $\alpha=\cos\theta\widetilde{\alpha}+\sin\theta\widetilde{\beta}$ and $\beta=-\sin\theta\widetilde{\alpha}+\cos\theta\widetilde{\beta}$, where $\theta=\frac{\pi}{3}$, Equations \eqref{integrability1} and \eqref{integrability2} imply
\begin{eqnarray*}
\alpha_s+\beta_t&=&-\frac{4}{\sqrt{3}}\alpha\times\beta\\
\alpha_t-\beta_s&=&0
\end{eqnarray*}

Note that in the special case $PTM\subset T^{\perp}M\,,$
we have that $\widetilde{\alpha}$ and $\widetilde{\beta}$ are orthogonal with respect to the induced Euclidean product metric and of the same length (this is also true for $\alpha$ and $\beta$).

Hence, from the relation between the nearly K\"{a}hler metric and the usual Euclidean product metric (see \eqref{g in function of <,> explicitly}), we have 
$$
g(F_s,F_s)=\langle\alpha,\alpha\rangle=g(F_t,F_t) \quad {\rm and}\quad g(F_s,F_t)=0\,.
$$
Since $s,t$ are the isothermal coordinates, we may assume $g(F_s,F_s)=g(F_t,F_t)=e^{2\omega}$ for a smooth positive function $\omega$ on $M$. Then, the induced metric on the surface $M$ is given by
$$
g=e^{2\omega}ds^2+e^{2\omega}dt^2\,.
$$
It follows that the Levi Civita connection on the surface $M$ is given by
\begin{eqnarray*}
\nabla_{F_s}F_s&=&\omega_s F_s-\omega_t F_t\\\notag
\nabla_{F_t}F_t&=&\omega_tF_t-\omega_sF_s\\\notag
\nabla_{F_s}F_t=\nabla_{F_t}F_s&=&\omega_tF_s+\omega_sF_t\,.
\end{eqnarray*}

From \eqref{first derivative} and \eqref{gamma tilde and delta tilde in function of alpha tilde and beta tilde}, we have that the components $A$ and $B$ of the immersion $F$ verify the equations
\begin{align}
\left\{\begin{aligned}
A_{ss}&=A(e^{2\omega}I+{\widetilde{\alpha}}_s)\,,\\
A_{st}&=A(\widetilde{\alpha}\widetilde{\beta}+{\widetilde{\beta}}_s)\,,\\
A_{tt}&=A(e^{2\omega}I+{\widetilde{\beta}}_t)\,,
\end{aligned}
\right.
\quad {\rm and}\quad
\left\{\begin{aligned}
B_{ss}&=B(e^{2\omega}I+\tfrac{1}{2}{\widetilde{\alpha}}_s-\tfrac{\sqrt{3}}{2}{\widetilde{\beta}}_s)\,,\\
B_{st}&=B(\tfrac{1}{4}(\widetilde{\alpha}\widetilde{\beta}-3\widetilde{\beta}\widetilde{\alpha})+\tfrac{\sqrt{3}}{2}{\widetilde{\alpha}}_s+\tfrac{1}{2}{\widetilde{\beta}}_s)\,,\\
B_{tt}&=B(e^{2\omega}I+\tfrac{\sqrt{3}}{2}{\widetilde{\alpha}}_t+\tfrac{1}{2}{\widetilde{\beta}}_t)\,.\\
\end{aligned}
\right.
\end{align}

So far the above equations remained valid for any almost complex surface for which $P$ maps to tangent space to the normal space. From now on we will assume that the surface is moreover totally geodesic. Then we have the following.
\begin{proposition}
A totally geodesic  almost complex surface $M$  of $SL(2,\mathbb{R})\times SL(2,\mathbb{R})$ for which $P$ maps tangent vectors into normal vectors   is isometric to the  immersion $(s,t)\mapsto (A(s,t),B(s,t))$ where
\begin{equation*}
A(s,t)=\left(
\begin{array}{cc}
 \frac{\sqrt{3}y_1}{2}+\frac{1}{2} & \frac{1}{2} \sqrt{3} (y_2+y_3)
   \\
 \frac{1}{2} \sqrt{3} (y_2-y_3) & \frac{1}{2}-\frac{\sqrt{3}y_1}{2}
   \\
\end{array}
\right)
\end{equation*}
and
\begin{equation*}
B(s,t)=\left(
\begin{array}{cc}
 \frac{1}{2}-\frac{\sqrt{3} y_1}{2} & -\frac{1}{2} \sqrt{3} (y_2+y_3)
   \\
 -\frac{1}{2} \sqrt{3} (y_2-y_3) & \frac{\sqrt{3} y_1}{2}+\frac{1}{2}
   \\
\end{array}
\right)\,,
\end{equation*}
where $(y_1,y_2,y_3)$ is a point of the hyperbolic quadric $$y_1^2+y_2^2-y_3^2=-1$$
\end{proposition}
\begin{proof}
Using $$D_X Y=\nabla^E_X Y +\tfrac{1}{2}\langle X,Y\rangle F+\tfrac{1}{2}\langle X,QY\rangle QF$$
which relates the pseudo Euclidean connection with the Levi Civita connection of the nearly kaehler metric and by supposing that $M$ is totally geodesic and then by using Proposition \ref{proposition G explicity} we have
\begin{align*}
\left\{\begin{aligned}
A_{ss}&=A(\omega_s\widetilde{\alpha}-\omega_t\widetilde{\beta}-\tfrac{1}{\sqrt{3}}\widetilde{\alpha}\times\widetilde{\beta}+ e^{2\omega}I)\,,\\
A_{st}&=A(\omega_t\widetilde{\alpha}+\omega_s\widetilde{\beta})\,,\\
A_{tt}&=A(-\omega_s\widetilde{\alpha}+\omega_t\widetilde{\beta}-\tfrac{1}{\sqrt{3}}\widetilde{\alpha}\times\widetilde{\beta}+e^{2\omega}I)\,,
\end{aligned}
\right.
\end{align*}
and
\begin{align*}
\left\{\begin{aligned}
B_{ss}&=B(\omega_s\widetilde{\gamma}-\omega_t\widetilde{\delta}-\tfrac{1}{\sqrt{3}}\widetilde{\alpha}\times\widetilde{\beta}+ e^{2\omega}I)\\
B_{st}&=A(\omega_t\widetilde{\gamma}+\omega_s\widetilde{\delta})\,,\\
B_{tt}&=B(-\omega_s\widetilde{\gamma}+\omega_t\widetilde{\delta}+\tfrac{1}{\sqrt{3}}\widetilde{\alpha}\times\widetilde{\beta}+e^{2\omega}I)\,.
\end{aligned}
\right.
\end{align*}

Hence, by identification we deduce that $\widetilde{\alpha}$ and $\widetilde{\beta}$ are determined by the following system of partial differential equations
\begin{eqnarray*}
\widetilde{\alpha}_s&=&-\omega_t\widetilde{\beta}+\omega_s\widetilde{\alpha}-\tfrac{1}{\sqrt{3}}\widetilde{\alpha}\times\widetilde{\beta}\,,\\
\widetilde{\alpha}_t&=&\omega_t\widetilde{\alpha}+\omega_s\widetilde{\beta}+\widetilde{\alpha}\times\widetilde{\beta}\,,\\
\widetilde{\beta}_s&=&\omega_t\widetilde{\alpha}+\omega_s\widetilde{\beta}-\widetilde{\alpha}\times\widetilde{\beta}\,,\\
\widetilde{\beta}_t&=&\omega_t\widetilde{\beta}-\omega_s\widetilde{\alpha}-\tfrac{1}{\sqrt{3}}\widetilde{\alpha}\times\widetilde{\beta}\,,
\end{eqnarray*}
which in terms of $\alpha$ and $\beta$ becomes
\begin{eqnarray*}
\alpha_s&=&-\omega_t\beta+\omega_s\alpha-\tfrac{2}{\sqrt{3}}\alpha\times\beta\,,\\
\alpha_t&=&\omega_t\alpha+\omega_s\beta\,,\\
\beta_s&=&\omega_t\alpha+\omega_s\beta\,,\\
\beta_t&=&\omega_t\beta-\omega_s\alpha-\tfrac{2}{\sqrt{3}}\alpha\times\beta\,.
\end{eqnarray*}

 Then locally there exists a matrix $\varepsilon$ such that $\varepsilon_s=\alpha$, $\varepsilon_t=\beta$. The surface $\varepsilon$ is determined by
\begin{eqnarray}\label{system of equations epsilon}
\varepsilon_{ss}&=&-\omega_t\varepsilon_t+\omega_s\varepsilon_s-\tfrac{2}{\sqrt{3}}\varepsilon_s\times\varepsilon_t\,,\notag\\
\varepsilon_{st}&=&\omega_t\varepsilon_s+\omega_s\varepsilon_t\,,\notag\\
\varepsilon_{tt}&=&\omega_t\varepsilon_t-\omega_s\varepsilon_s-\tfrac{2}{\sqrt{3}}\varepsilon_s\times\varepsilon_t\,.
\end{eqnarray}
A unit normal vector field on the surface $M$ determined by the matrix $\varepsilon$ is given by $\xi=-\frac{\varepsilon_s\times\varepsilon_t}{e^{2\omega}}$

On the other hand, $\varepsilon_{st}=D_{\varepsilon_s}\varepsilon_t=\nabla_{\varepsilon_s}\varepsilon_t+\widetilde{h}(\varepsilon_s,\varepsilon_t)\xi$, by identification with the second equation of \eqref{system of equations epsilon}, we deduce that $h(\varepsilon_s,\varepsilon_t)=0$, in a similar way $h(\varepsilon_s,\varepsilon_s)=h(\varepsilon_t,\varepsilon_t)=\frac{2}{\sqrt{3}}e^{2\omega}\xi\,.$
Hence the surface determined by $\varepsilon$ is totally umbilical with shape operator $S=\frac{2}{\sqrt{3}}I\,.$
Moreover, as $D_{\frac{\partial}{\partial_u}}(\varepsilon+\frac{\sqrt{3}}{2}\xi)=D_{\frac{\partial}{\partial_v}}(\varepsilon+\frac{\sqrt{3}}{2}\xi)=0$, we deduce that $\varepsilon$ is a hyperbolic space of cartesian equation 
$$\varepsilon_1^2+\varepsilon_2^2-\varepsilon_3^2=-\frac{3}{4}\,.$$ 
From this we can now reverse engineer the original immersion. We can take a local isothermal parametrisation of this hypersurface by
\begin{align*}
&\varepsilon_1=-\tfrac{\sqrt{3} s}{(s^2 + t^2 - 1)}\\
&\varepsilon_2=-\tfrac{\sqrt{3} t}{(s^2 + t^2 - 1)}\\
&\varepsilon_3=\tfrac{\sqrt{3}}{2}\tfrac{s^2+t^2+1}{s^2+t^2-1}
\end{align*}
where we look at $\varepsilon$ as the matrix 
$$\varepsilon=\begin{pmatrix} \varepsilon_1 & \varepsilon_2+\varepsilon_3\\
\varepsilon_2-\varepsilon_3 &\varepsilon_1
\end{pmatrix}$$
It then follows immediately that 
\begin{align*}
\alpha=\left(
\begin{array}{cc}
 \frac{\sqrt{3} \left(s^2-t^2+1\right)}{\left(s^2+t^2-1\right)^2} & \frac{2 \sqrt{3} s
   (t-1)}{\left(s^2+t^2-1\right)^2} \\
 \frac{2 \sqrt{3} s (t+1)}{\left(s^2+t^2-1\right)^2} & \frac{\sqrt{3}
   \left(-s^2+t^2-1\right)}{\left(s^2+t^2-1\right)^2} \\
\end{array}
\right)\\
\beta=\left(
\begin{array}{cc}
 \frac{2 \sqrt{3} s t}{\left(s^2+t^2-1\right)^2} & \frac{\sqrt{3}
   \left((t-1)^2-s^2\right)}{\left(s^2+t^2-1\right)^2} \\
 \frac{\sqrt{3} \left((t+1)^2-s^2\right)}{\left(s^2+t^2-1\right)^2} & -\frac{2 \sqrt{3} s
   t}{\left(s^2+t^2-1\right)^2} \\
\end{array}
\right)\,.
\end{align*}
From this we now deduce that
\begin{align*}
&\widetilde \alpha=\left(
\begin{array}{cc}
 \frac{\sqrt{3} s^2-6 s t-\sqrt{3} \left(t^2-1\right)}{2 \left(s^2+t^2-1\right)^2} &
   \frac{3 s^2+2 \sqrt{3} s (t-1)-3 (t-1)^2}{2 \left(s^2+t^2-1\right)^2} \\
 \frac{3 s^2+2 \sqrt{3} s (t+1)-3 (t+1)^2}{2 \left(s^2+t^2-1\right)^2} & \frac{-\sqrt{3}
   s^2+6 s t+\sqrt{3} \left(t^2-1\right)}{2 \left(s^2+t^2-1\right)^2} \\
\end{array}
\right)\\
&\widetilde \beta=\left(
\begin{array}{cc}
 \frac{3 s^2+2 \sqrt{3} s t-3 t^2+3}{2 \left(s^2+t^2-1\right)^2} & \frac{-\sqrt{3} s^2+6 s
   (t-1)+\sqrt{3} (t-1)^2}{2 \left(s^2+t^2-1\right)^2} \\
 \frac{-\sqrt{3} s^2+6 s (t+1)+\sqrt{3} (t+1)^2}{2 \left(s^2+t^2-1\right)^2} & -\frac{3
   s^2+2 \sqrt{3} s t-3 t^2+3}{2 \left(s^2+t^2-1\right)^2} \\
\end{array}
\right)\\
&\widetilde \gamma=\left(
\begin{array}{cc}
 \frac{-\sqrt{3} s^2-6 s t+\sqrt{3} \left(t^2-1\right)}{2 \left(s^2+t^2-1\right)^2} &
   \frac{3 s^2-2 \sqrt{3} s (t-1)-3 (t-1)^2}{2 \left(s^2+t^2-1\right)^2} \\
 \frac{3 s^2-2 \sqrt{3} s (t+1)-3 (t+1)^2}{2 \left(s^2+t^2-1\right)^2} & \frac{\sqrt{3}
   s^2+6 s t-\sqrt{3} \left(t^2-1\right)}{2 \left(s^2+t^2-1\right)^2} \\
\end{array}
\right)\\
&\widetilde \delta=\left(
\begin{array}{cc}
 \frac{3 s^2-2 \sqrt{3} s t-3 t^2+3}{2 \left(s^2+t^2-1\right)^2} & \frac{\sqrt{3} s^2+6 s
   (t-1)-\sqrt{3} (t-1)^2}{2 \left(s^2+t^2-1\right)^2} \\
 \frac{\sqrt{3} s^2+6 s (t+1)-\sqrt{3} (t+1)^2}{2 \left(s^2+t^2-1\right)^2} & \frac{-3
   s^2+2 \sqrt{3} s t+3 t^2-3}{2 \left(s^2+t^2-1\right)^2} \\
\end{array}
\right)\,.
\end{align*}
Solving now the differential equations for $A$ and $B$ we find that
\begin{align*}
A=\left(
\begin{array}{cc}
 \frac{1}{2}-\frac{\sqrt{3} s}{s^2+t^2-1} & \frac{\sqrt{3} \left(s^2+(t-1)^2\right)}{2
   \left(s^2+t^2-1\right)} \\
 -\frac{\sqrt{3} \left(s^2+(t+1)^2\right)}{2 \left(s^2+t^2-1\right)} & \frac{\sqrt{3}
   s}{s^2+t^2-1}+\frac{1}{2} \\
\end{array}
\right)\\
B=\left(
\begin{array}{cc}
 \frac{\sqrt{3} s}{s^2+t^2-1}+\frac{1}{2} & -\frac{\sqrt{3} \left(s^2+(t-1)^2\right)}{2
   \left(s^2+t^2-1\right)} \\
 \frac{\sqrt{3} \left(s^2+(t+1)^2\right)}{2 \left(s^2+t^2-1\right)} &
   \frac{1}{2}-\frac{\sqrt{3} s}{s^2+t^2-1} \\
\end{array}
\right)\,.
\end{align*}
Replacing now the coordinates by $\varepsilon_i$ and rescaling completes the proof of the proposition. 
\end{proof}

Finally we will study the case when $K=-5/9$ and $PTM\subset T^\perp M\,.$ We consider the basis
$$
e_1=V\quad e_2=JV \quad e_3=PV \quad e_4=JPV \quad e_5=G(V,PV) \quad e_6=-G(JV,PV)\,,
$$
where $V$ is a unit tangent vector. All these vectors are mutually orthogonal, $e_1,e_2,e_3,e_4$ are of unit length and $e_5,e_6$ are of length $-2/3\,.$

By using the relation \eqref{G(X,G(Z,W))} we have
\small{
\begin{align*}
G(e_1,e_2)&=0& G(e_1,e_3)&=e_5& G(e_1,e_4)&=-e_6 & G(e_1,e_5)&=\tfrac{2}{3}e_3 & G(e_1,e_6)&=-\tfrac{2}{3}e_4\\
G(e_2,e_3)&=-e_6& G(e_2,e_4)&=-e_5& G(e_2,e_5)&=-\tfrac{2}{3}e_4& G(e_2,e_6)&=-\tfrac{2}{3}e_3& G(e_3,e_4)&=0\\
G(e_3,e_5)&=\tfrac{2}{3}e_1& G(e_3,e_6)&=\tfrac{2}{3}e_2& G(e_4,e_5)&=\tfrac{2}{3}e_2& G(e_4,e_6)&=\tfrac{2}{3}e_1& G(e_5,e_6)&=0\,.
\end{align*}
}

Let $\widetilde{\nabla}$ denote the Levi-Civita connection on ${\rm SL}(2,\mathbb{R})\times {\rm SL}(2,\mathbb{R})\,.$
We will write $\widetilde{\nabla}_{e_1}e_1=\nabla_{e_1}e_1+h(e_1,e_1)=a_1e_2+a_2e_5+a_3e_6$ (the coefficients of $e_3$ and $e_4$ are $0$ from page 6 in \cite{Bo-Di-Di-Vr}), where $a_1,a_2,a_3$ are functions to determine and let $b_1=g(\widetilde{\nabla}_{e_2}e_1,e_2)$.

The Levi-Civita connection $\widetilde{\nabla}$ is computed in term of $a_1,a_2,a_3$ and $b_1$, more precisely

\begin{eqnarray*}
\widetilde{\nabla}_{e_1}e_2&=&-a_1e_1-a_3e_5+a_2e_6\\
\widetilde{\nabla}_{e_1}e_3&=&-a_1e_4+a_2e_5+(\tfrac{1}{2}-a_3)e_6\\
\widetilde{\nabla}_{e_1}e_4&=&a_1e_3+(\tfrac{1}{2}+a_3)e_5+a_2e_6\\
\widetilde{\nabla}_{e_1}e_5&=&\tfrac{2}{3}a_2e_1-\tfrac{2}{3}a_3e_2+\tfrac{2}{3}a_2e_3+\tfrac{2}{3}(\tfrac{1}{2}+a_3)e_4\\
\widetilde{\nabla}_{e_1}e_6&=&\tfrac{2}{3}a_3e_1+\tfrac{2}{3}a_2e_2+\tfrac{2}{3}(\tfrac{1}{2}-a_3)e_3+\tfrac{2}{3}a_2e_4\\
\widetilde{\nabla}_{e_2}e_1&=&b_1e_2-a_3e_5+a_2e_6\\
\widetilde{\nabla}_{e_2}e_2&=&-b_1e_1-a_2e_5-a_3e_6\\
\widetilde{\nabla}_{e_2}e_3&=&-b_1e_4+(\tfrac{1}{2}-a_3)e_5-a_2e_6\\
\widetilde{\nabla}_{e_2}e_4&=&b_1e_3+a_2e_5-(\tfrac{1}{2}+a_3)e_6\\
\widetilde{\nabla}_{e_2}e_5&=&-\tfrac{2}{3}a_3e_1-\tfrac{2}{3}a_2e_2+\tfrac{2}{3}(\tfrac{1}{2}-a_3)e_3+\tfrac{2}{3}a_2e_4\\
\widetilde{\nabla}_{e_2}e_6&=&\tfrac{2}{3}a_2e_1-\tfrac{2}{3}a_3e_2-\tfrac{2}{3}(\tfrac{1}{2}+a_3)e_4-\tfrac{2}{3}a_2e_3\,.
\end{eqnarray*}
From the Gauss equation, we have that the Gauss curvature $K=\langle R(e_1,e_2)e_2,e_1\rangle$ is given by \begin{eqnarray}
K&=&\langle \widetilde{R}(e_1,e_2)e_2,e_1\rangle-\langle h((e_1,e_2),h(e_1,e_2)\rangle+\langle h(e_2,e_2),h(e_1,e_1)\rangle
\end{eqnarray}
which implies that
\begin{eqnarray*}
-\frac{5}{9}&=&-\frac{4}{3}+\frac{4}{3}(a_2^2+a_3^2)
\end{eqnarray*}
Then
\begin{equation}\label{contradiction}
a_2^2+a_3^2=\frac{7}{12}
\end{equation}
Now by applying that $\widetilde{R}(e_1,e_2).=\widetilde{\nabla}_{e_1}\widetilde{\nabla}_{e_2}X-\widetilde{\nabla}_{e_2}\widetilde{\nabla}_{e_1}X-\widetilde{\nabla}_{[e_1,e_2]}X$ for $X=e_1,\ldots,e_6$, with $[e_1,e_2]=-a_1e_1-b_1e_2$, we deduce the following equations
\begin{align*}
\left\{
\begin{aligned}
-3(a_1^2+b_1^2-e_2(a_1)+e_1(b_1))+\tfrac{5}{3}&=0\\
2(a_1a_3+a_2b_1)+e_1(a_2)-e_2(a_3)&=0\\
2(a_1a_2-a_3b_1)-e_2(a_2)-e_1(a_3)&=0
\end{aligned}
\right.
\end{align*}
If the surface is parallel, $\nabla h=0$ gives the following equations
\begin{align*}
\left\{
\begin{aligned}
\tfrac{2}{3}a_2^2+\tfrac{1}{3}(1-2a_3)a_3&=0\\
\tfrac{2}{3}a_2a_3+\tfrac{1}{3}(1+2a_3)a_2&=0\\
2a_1a_3+e_1(a_2)&=0\\
-2a_1a_2+e_1(a_3)&=0
\end{aligned}
\right.
\end{align*}
This system has a unique null solution, which is in contradiction with $\eqref{contradiction}$. Thus no solutions exist. Then Proposition \ref{prop-parallel-second-fundamental-form} becomes

\begin{theorem}
If $M$ is an almost complex surface in $\nks$ with parallel second fundamental form, then either
\begin{enumerate}
\item P maps the tangent space into the normal space and in this case, $M$ is totally geodesic with constant Gaussian curvature $-\dfrac{4}{3}$. Moreover, $M$ is totally geodesic and congruent to the example described in Proposition 4.5. 
\item P preserves the tangent space, in this case the Gaussian curvature is $0$. Moreover $M$ is totally geodesic and  congruent to either the example described in Proposition 4.1 or the example described in Proposition 4.2. 
\end{enumerate}
\end{theorem}
This means every parallel almost complex surface in $\nks$ is totally geodesic.

\end{document}